\documentclass[12pt]{amsart} 
 \usepackage{amsmath} \usepackage{amsmath, pb-diagram} \usepackage[all]{xy}
\usepackage{amssymb} \usepackage{amscd} \newtheorem{thm}{Theorem}[section] 
\newtheorem{lem}[thm]{Lemma} \newtheorem{prop}[thm]{Proposition} \theoremstyle{definition}
  \theoremstyle{remark} \newtheorem{rem}[thm]{Remark}
\numberwithin{equation}{section}  
  
\begin{document}

\author[Cristian Cazacu]{Cristian Cazacu}
\address{Cristian Cazacu
\hfill\break\indent
Faculty of Mathematics and Computer Science \&
\hfill\break\indent  The Research Institute of the University of Bucharest
\hfill\break\indent University of Bucharest 
\hfill\break\indent 14 Academiei Street 
\hfill\break\indent 010014 Bucharest, Romania.}
\hfill\break\indent \email{cristian.cazacu@fmi.unibuc.ro}

\thanks{ Correspondence address: Cristian
Cazacu, Department of Mathematics, Faculty of Mathematics and Computer Science, University of Bucharest,
010014 Bucharest, Romania. E-mail: {\tt cristian.cazacu@fmi.unibuc.ro}}

\title[An overview of Hardy and Rellich type inequalities]{The method of super-solutions in Hardy and Rellich type inequalities in the $L^2$ setting: an overview of well-known results and short proofs}

\maketitle
\begin{abstract} 
In this survey we give a compact presentation of  well-known functional inequalities of Hardy and Rellich type  in the $L^2$ setting. In addition, we give some insights of their proofs by using standard and basic tools such as the method of super-solutions. \\	 
{\it 2020 Mathematics Subject Classification:} 35A23, 35R45, 35Q40, 35B09, 34A40, 34K38\\
{\it Key words:} Hardy inequalities; Rellich inequalities; optimal constants; super-solutions.  
 \end{abstract}

\section{Introduction}

This work is aimed to be an overview devoted  to well-known functional inequalities of Hardy and Rellich type in the $L^2(\mathbb{R}^d)$-setting, in the presence of singular potentials $V$ with singularities of the form $1/|x|^{\alpha}$, where $\alpha>0$ will be well precised later.  By $|x|$ we understand the euclidian norm of a vector $x\in \mathbb{R}^d$.

 We have two main objectives: i) to present one of the most famous results in the literature related to the subject; ii) to provide some classical  techniques which give rise to short proofs of the quoted results. 

Meeting the objectives could be extremely useful \textit{"at a first glance"} for readers who are not very familiarized with this topic. 

The literature on Hardy and Rellich type inequalities is very vast and therefore we will try to pick up the most common situations. We focus on inequalities in smooth domains and on singular potentials $V(x)=|x|^{\alpha}$  with critical homogeneity (which is $\alpha=-2$ in the case of standard Hardy inequality and $\alpha=-4$ for the standard Rellich inequality) as we will see later on.

\subsection{A piece of history of the Hardy inequality}\label{s1}
The history of the famous Hardy inequality has about 100 years.  It was in the 1920's when Godfrey Harold Hardy answered to a discrete inequality of David Hilbert with a new inequality which was also discrete, asserting that: for any $p>1$ and the positive numbers $a_i$, with $i=1, n$, such that $\sum_{n\geq 1}a_n^p$ is convergent then $\sum_{n\geq 1} A_n^p$ is also convergent (where $A_n=\sum_{i=1}^n a_i/n$) and it holds 
\begin{equation}\label{discrete Hardy}
\sum_{n=1}^{\infty} A_{n}^p \leq \left(\frac{p}{p-1}\right)^p \sum_{n=1}^{\infty} a_n^p.
\end{equation}
It is worth mentioning that \eqref{discrete Hardy} was initially stated in \cite{H1} in a weaker form, i.e. with a higher constant on the right hand side, namely  $\left(\frac{p^2}{p-1}\right)^p$. In the same paper Hardy stated the  continuous version (in the $L^p$-setting) of the above inequality, which is
\begin{equation}\label{integral Hardy}
\int_{a}^\infty \left(\frac{F(x)}{x}\right)^p dx \leq \left(\frac{p}{p-1}\right)^p \int_{a}^{\infty} f(x)^p dx, \quad \textrm{ with } F(x)=\int_{a}^x f dt,
\end{equation}
\textit{"for any $a$ and $f$ positive"}, said Hardy, avoiding to say something about the regularity, integrability or asymptotic behavior of the admissible function $f$. Basically, these results and auxiliary extensions  appeared  in the works \cite{H1, H2} and they were highlighted later in \cite{HLP}. Much more details about the origins  of the Hardy inequality and its first developments can be found in the book \cite{K1}.  The modern $L^p$ version of the 1-d Hardy inequality \eqref{integral Hardy} states that 
\begin{equation}\label{integral Hardy}
\int_{0}^\infty \left(\frac{F(x)}{x}\right)^p dx \leq \left(\frac{p}{p-1}\right)^p \int_{0}^{\infty} |F'(x)|^p dx, \quad  \forall F\in C_c^\infty(0, \infty), 
\end{equation}
which could be easily extended by Fatou Lemma to test functions $F$ in the Sobolev space $W_0^{1, p}(0, \infty)$.

 Since 1920's the Hardy inequality turned out to represent one of the most important tools in the analysis of Partial Differential Equations (PDE). For instance, Leray, in his celebrated paper \cite{L1} from 1934 when he studied the well-posedness of the weak solutions of the Navier-Stokes equation,  applied  a Hardy-type inequality involving partial derivatives:
 
 \begin{equation}\label{Hardy_Leray}
  \int_{\mathbb{R}^3} |\nabla u|^2 dx \geq \frac{1}{4}\int_{\mathbb{R}^3} \frac{u^2}{|x|^2} dx, \quad \forall u\in C_c^\infty(\mathbb{R}^3). 
 \end{equation}  
This is the extension of the $L^2$-version of \eqref{integral Hardy} to three dimensions. More general, \eqref{Hardy_Leray} can be extended to any dimension $d\geq 3$: 
\begin{equation}\label{standard HI}
\int_{\mathbb{R}^d} |\nabla u|^2 dx \geq \left(\frac{d-2}{2}\right)^2 \int_{\mathbb{R}^d} \frac{u^2}{|x|^2} dx, \quad \forall u\in C_0^\infty(\mathbb{R}^d),
\end{equation} 
where $\left(\frac{d-2}{2}\right)^2$ is the biggest admissible constant, i.e. the optimal one. This is what we usually call nowadays the \textit{standard/classical Hardy inequality} in $L^2$ form. 

\subsection{The Rellich inequality}

This started with the pioneering work of Franz Rellich \cite{R1} in the '50s when his inequality was first published in print form. The \textit{Rellich inequality} states that   
	\begin{equation}\label{Rellich}
\int_{\mathbb{R}^d} |\Delta u|^2 dx \geq \frac{d^2(d-4)^2}{16}\int_{\mathbb{R}^d}  \frac{u^2}{|x|^4} dx, \quad \forall u\in C_c^\infty(\mathbb{R}^d),
\end{equation}
 It seems that Rellich proved it earlier and delivered it in some lectures at New York University in 1953 which were published posthumously in \cite{R2}. 
 The Rellich inequality \eqref{Rellich} has undergone extensive further developments beginning with the $L^p$ version in \cite{DH}, in particular, being an important tool to study spectrum of biharmonic-type operators.

 \subsection{The Hardy-Rellich inequality} 
  The so-called classical {\it Hardy-Rellich inequality} is a  mixture of both inequalities \eqref{standard HI} and \eqref{Rellich}. It appeared as a consequence of the Hardy inequality applied to special classes of vector fields deriving from a potential gradient. More precisely, if we apply \eqref{standard HI} for each component of the potential gradient $\vec{U}=\nabla u$ and sum up all the terms we get 
  \begin{equation}\label{First HR inequality}
  \int_{\mathbb{R}^d} |\Delta u|^2 dx \geq \left(\frac{d-2}{2}\right)^2 \int_{\mathbb{R}^d} \frac{|\nabla u|^2}{|x|^2} dx, \quad \forall u\in C_c^\infty(\mathbb{R}^d).
  \end{equation}
  This is a consequence of the fact that after integrating by parts twice  we obtain 
  	 \begin{align*}
  \sum_{i=1}^d \int_{\mathbb{R}^d} |\nabla \partial_{x_i} u|^2 dx &=\sum_{i,j=1}^{d} \int_{\mathbb{R}^d} |\partial_{x_ix_j}^2 u|^2 dx=\int_{\mathbb{R}^d} |\Delta u|^2 dx. 
  \end{align*}   
   Notice that this is related to divergence free vector fields so important in fluid mechanics since in 3-d we have $\textrm{div} \vec{U}=0$. As we will see later the constant $\left(\frac{d-2}{2}\right)^2$ in \eqref{First HR inequality} is not optimal.

\section{The Hardy inequality in $L^2$ setting: basic results and proofs}
Next in the paper we will denote by $\Omega$ a smooth domain (open and connected set) in $\mathbb{R}^d$.  
In this section we mainly discuss the Hardy inequality 
\begin{equation}\label{Hardy_goal}
\mu>0, \quad  \int_\Omega |\nabla u|^2 dx \geq \mu \int_\Omega V u^2 dx, \quad \forall u\in C_0^\infty(\Omega),
\end{equation}
and analyze the biggest admissible constant $\mu$ in \eqref{Hardy_goal}, i.e. the \textit{optimal/best Hardy constant}  which is defined by 
\begin{equation}\label{optim1}
\mu^\star(V, \Omega):=\inf_{u\in C_{c}^{\infty}(\Omega)\setminus \{0\}}\frac{\int_{\Omega}|\nabla u|^2 dx}{\int_{\Omega} V u^2 dx}
\end{equation}  
We focus on two type of potentials:
\begin{enumerate}
	\item \footnote{For the sake of clarity we fix the singularity at the origin $x=0$  but all the forthcoming results are valid for more general potentials of the form $V(x)=1/|x-a|^2$, with the singularity located at an arbitrary point $a\in \mathbb{R}^d$.}   one singular inverse-square potential $V(x)=1/|x|^2$;
	\item  multipolar potentials of the form $V(x)= \sum_{1\leq i< j\leq n} \frac{|a_i-a_j|^2}{|x-a_i|^2|x-a_j|^2}$ with finite number of singular poles $a_1, \ldots, a_n\in \mathbb{R}^d$.
\end{enumerate}
I turns out that in general the best constant $\mu^\star(V, \Omega)$ depends on the structure of the potential $V$  and the geometry of $\Omega$. However, when there is no risk of confusion, we will write $\mu^\star$ or $\mu^\star(\Omega)$ instead of $\mu^\star(V, \Omega)$. 

We also analyze the attainability of $\mu^\star$ in \eqref{optim1}.  We will see that $C_c^\infty(\Omega)$ is not a good functional space for seeking for minimizers of $\mu^\star$ but  the \textit{energy space} $\mathcal{D}^{1,2}(\Omega)$ which is naturally defined as the completion of $C_0^\infty(\Omega)$ in the energy norm  $$\|u\|:=\int_\Omega |\nabla u|^2 dx.$$
It is straightforward that $\mu^\star$ in \eqref{optim1} can  also  be characterized by  
$$\mu^\star(V, \Omega)=\inf_{u\in \mathcal{D}^{1, 2}(\Omega)\setminus \{0\}}\frac{\int_{\Omega}|\nabla u|^2 dx}{\int_{\Omega} V u^2 dx}.$$  
The space $\mathcal{D}^{1,2}(\Omega)$ is the largest functional space where  inequality \eqref{Hardy_goal} makes sense. Clearly, $\mathcal{D}^{1,2}(\Omega)=H_0^1(\Omega)$ if $\Omega$ is a domain for which the Poincar\'{e} inequality applies (such as bounded domains) otherwise the inclusion $\mathcal{D}^{1,2}(\Omega)  \subset H_0^1(\Omega)$ is strict as it happens for the whole space $\mathbb{R}^d$.

Thus, it is then obvious that \eqref{Hardy_goal} holds in the range $\mu\leq  \mu^\star$. In other words the Hardy inequality \eqref{Hardy_goal}-\eqref{optim1} is equivalent to  
 $$H:=-\Delta -\mu V\geq 0, \quad \forall \mu\leq \mu^\star,$$
which means that the Hamiltonian $H$ is nonnegative in the sense of $L^2$ quadratic forms, i.e. $(Hu, u)_{L^2(\Omega)}\geq 0$, for all $u\in C_c^\infty(\Omega)$, where $(\cdot, \cdot)_{L^2(\Omega)}$ denotes the scalar product in $L^2(\Omega)$.   
 
In  applications Hardy inequalities are in general  
both important and difficult. In the spectral theory
of (magnetic) differential elliptic operators
it is of capital importance to obtain sharp lower
bounds for the corresponding quadratic forms in
order to control local singularities induced by
different perturbations.The boundedness from below for the self-adjoint extension of a symmetric operator of the form $H$ means 
$H\geq c$, where $c$ is a real constant, not necessary positive.   If $c<0$, writing in quadratic forms this is equivalent to the {\it weak} Hardy inequality  
\begin{equation}\label{weak Hardy}
 \int_\Omega |\nabla u|^2 dx \geq \mu \int_\Omega V u^2 dx+c \int_\Omega u^2 dx, \quad \forall u\in C_0^\infty(\Omega).
\end{equation}
In many circumstances the weaker inequality \eqref{weak Hardy} may suffices because the $L^2$ lower order term could be eventually absorbed so that it does not influence the main result of the problem.    In terms of quantum mechanics boundedness from below is related to the stability of matter.   
   For instance, in the case of the simplified Coulomb potential, we can obtain the lower bound (which is actually sharp)
$$-\Delta - \frac{Z}{|x|} \geq  -\frac{Z}{(d-2)^2}, \textrm{ in } L^2(\mathbb{R}^d),$$
due to the Hardy inequality \eqref{standard HI} as follows: 
\begin{align*}
-\Delta -\frac{Z}{|x|}&=\underbrace{-\Delta -\frac{\mu^\star}{|x|^2}}_{\geq 0}+ \frac{\mu^\star}{|x|^2}-\frac{Z}{|x|}\geq  -\frac{Z}{4\mu^\star}; \quad \mu^\star=\frac{(d-2)^2}{4}.   
\end{align*}

Important applications of Hardy inequalities in PDE appear in the well-posedness or regularity theory of solutions and especially for those PDE with singular perturbations. Also, they play a crucial role in the theory of function spaces, see e.g. \cite{Petru}. Other applications could be consulted for instance in  \cite{PhDCristi}, \cite{BV}.    
 
\subsection{The case of $V(x)=\frac{1}{|x|^2}$} In this case  we distinguish two situations for the Hardy inequality \eqref{Hardy_goal}-\eqref{optim1} with respect to the location of the singularity: 
\begin{itemize} 
	\item $x=0$ is located in $\Omega$ (interior singularity)
	\item $x=0$ is located on $\partial \Omega$ (boundary singularity)
\end{itemize} 

 \begin{thm}[Interior singularity]\label{th1} Let $d\geq 3$ and assume that $0\in \Omega$. Then 
\begin{equation}\label{classical_Hardy_inequality}
\int_\Omega |\nabla u|^2 dx \geq \mu^\star \int_\Omega \frac{u^2}{|x|^2} dx, \quad \forall u\in C_0^\infty(\Omega),
\end{equation}
and 
\begin{equation}\label{global_Hardy}
\mu^\star(\Omega)=\mu^\star(\mathbb{R}^d)=\frac{(d-2)^2}{4}.
\end{equation}
Moreover, $\mu^\star$ is  never attained in the energy space $\mathcal{D}^{1, 2}(\Omega)$.  
\end{thm}
It is worth to mention that the singular term $\int_\Omega u^2/|x|^2 dx$ in \eqref{classical_Hardy_inequality} is finite due to the fact that $1/|x|^2\in L_{loc}^1(\mathbb{R}^d)$ as long as $d\geq 3$. 

The statement of Theorem \ref{th1} appears in a form or another in the majority of papers on this topic. A reference paper which subsequently created a lot of interest and developments on the subject is \cite{BV}.   

There are many proofs of \eqref{classical_Hardy_inequality} in Theorem \ref{th1} but  here we present a very simple one.
\begin{proof}[A ''two lines" proof of  \eqref{classical_Hardy_inequality}] Applying one integration by parts  and the Cauchy-Schwarz inequality we successively obtain
	\begin{align*}
	\int_\Omega \frac{u^2}{|x|^2} dx &= \frac{1}{d-2} \int_\Omega \textrm{div}\left(\frac{x}{|x|^2} \right) u^2 dx =-\frac{2}{d-2} \int_\Omega \frac{u}{|x|^2} x\cdot \nabla  u dx\\
	&\leq \frac{2}{d-2}\left(\int_\Omega \frac{u^2}{|x|^2} dx \right)^{1/2} \left(\int_\Omega |\nabla u|^2 dx \right)^{1/2}.
\end{align*}
	Taking the squares in the extreme terms above and simplifying we precisely get \eqref{classical_Hardy_inequality}. The same proof works in the whole space $\mathbb{R}^d$ when working with smooth  compactly supported functions. 
	\end{proof}
\begin{rem}
	If we want to be totally rigorous in the "two lines proof" we need to avoid the singularity when doing integration by parts. An alternative option for that is to "regularize" the potential and to mimic the above proof starting with the term $\int_{\Omega} \frac{u^2}{|x|^2+\epsilon^2}$ and then pass to the limit as $\epsilon\searrow 0$. We let the details to the reader.   
\end{rem}
 Of course, from  \eqref{classical_Hardy_inequality} we obviously have $\mu^\star(\Omega)\geq (d-2)^2/4$ and $\mu^\star(\mathbb{R}^d)\geq (d-2)^2/4$. In order to show \eqref{optim1} we need to design a minimizing sequence to approach the constant $(d-2)^2/4$.  Let $\epsilon>0$ and let $R>0$ be such that the ball of radius $2R$ centered at the origin, denoted by $B_{2R}(0)$, is a subset of  $\Omega$.  We define
 \begin{equation}\label{min_seq}
 u_\epsilon(x)=(|x|^{2}+\epsilon^2)^{-\frac{d-2}{4}}\theta(|x|), 
 \end{equation} 
 where $\theta:[0, \infty)\rightarrow [0,1]$ is a smooth cut-off function such that $\theta(r)=1$ for $r\in [0, R]$ and $\theta(r)=0$ for $r\geq 2R$.  Obviously, $u_\epsilon\in C_c^\infty(\Omega)$. Moreover, a useful exercise shows that 
\begin{equation}\label{min_lim_seq}
 \frac{\int_{\Omega}|\nabla u_\epsilon|^2 dx}{\int_{\Omega}u_\epsilon^2/|x|^2 dx} \searrow \frac{(d-2)^2}{4}, \textrm{ as } \epsilon\searrow 0.
 \end{equation}
 Therefore, \eqref{classical_Hardy_inequality} and \eqref{min_lim_seq} imply \eqref{global_Hardy} (since a  function  $u \in C_c^\infty(\Omega)$ could be trivially extended to a function $\overline{u}\in C_c^\infty(\mathbb{R}^d)$). 
 
 The inconvenient of the "two lines" proof is the loss of evidence concerning the attainability of the best constant. However, one can notice that the difference of the terms in \eqref{classical_Hardy_inequality} could be written as (after integration by parts)
 \begin{align}\label{H_impr}
 \int_{\Omega} |\nabla u|^2 dx - \frac{(d-2)^2}{4} \int_{\Omega} \frac{|u|^2}{|x|^2} dx &=\int_{\Omega} \left|\nabla u+\frac{d-2}{2} \frac{x}{|x|^2} u \right|^2 dx\nonumber\\
 &=\int_{\Omega} \left|\nabla \left(u |x|^{\frac{d-2}{2}}\right)\right|^2 |x|^{2-d} dx,  \quad \forall u \in C_c^\infty(\Omega) 
 \end{align} 
 By density, \eqref{H_impr} can be extended to functions in $\mathcal{D}^{1, 2}(\Omega)$. If the constant $\mu^\star=(d-2)^2/4$ were  attained by a function $u\in \mathcal{D}^{1, 2}(\Omega)$ then it should satisfy \eqref{H_impr}. In that case, the left hand side is zero and therefore $u=C|x|^{-(d-2)/2}$ for some real constant $C$. However, one can check that $u\not \in \mathcal{D}^{1, 2}(\Omega)$ (both terms in \eqref{classical_Hardy_inequality} are infinite). Contradiction. So, the constant is not attained.

\begin{rem}
	When $d=2$  then \eqref{classical_Hardy_inequality} has no sense because $\mu^\star=0$ and the singular term is not integrable in general because $1/|x|^2\not\in L_{loc}^1(\mathbb{R}^2)$) (on the right hand side we may have the nedetermination $0\cdot \infty$).
\end{rem} 

\begin{rem}
	Observe that the "guess" of the minimizing sequence in \eqref{min_seq} is related to the singular function $|x|^{-(d-2)/2}$ found when discussing the attainability of the best constant $\mu^\star$ above. Indeed, $u_\epsilon$ is a regularization of $|x|^{-(d-2)/2}$ because $u_\epsilon$ converges  to $|x|^{-(d-2)/2}$ as $\epsilon$ tends to zero,  pointwise (except the origin) in a small neighborhood of the origin.
\end{rem}
\begin{rem}
	Identity \eqref{H_impr} seems to be "magical" and difficult to find, but in fact it has to do with the fact that the function $|x|^{-(d-2)/2}$ is a  solution of the operator $H:=-\Delta -\mu^\star/|x|^2$ associated to the quadratic form induced by the Hardy inequality \eqref{classical_Hardy_inequality}. More exactly we have, 
	\begin{equation}\label{extremal functions}
	\phi(x)=|x|^{-\frac{d-2}{2}} \textrm{ verifies } \underbrace{-\Delta \phi -\mu^\star\frac{\phi}{|x|^2}}_{=H\phi}=0, \quad \forall x\neq 0.   
	\end{equation}  
\end{rem}

In view of the remarks above next we give the statement of a more general result, the so-called the {\it method of super-solutions}, which works to prove large variety of functional inequalities, particularly both optimal Hardy and Rellich inequalities. 

\subsection{The method of super-solutions}

Rougly speaking it says that {\it "a second order elliptic operator $H$ is nonnegative in a domain $\Omega$ if there exists a positive super-solution solution $\phi$ in $\Omega$ for $H \phi\geq 0$"} (see,  e.g. \cite{A, Al, P}).  To be more specific, for our necessities in this section we need the following lemma 
\begin{lem}[adapted from \cite{Davies}]\label{l1}
	 Let $\phi$ be a positive function in $\Omega$ with $\phi \in C^2(\Omega\setminus\{0\})$ and  let $W\in L_{loc}^1(\Omega)$ be a continuous function 
	on $\Omega\setminus\{0\}$ such that	
	\begin{equation}\label{id1}(-\Delta - W )\phi(x) \geq 0, \quad \forall x\in \Omega\setminus\{0\}.
	\end{equation}
Then, $-\Delta - W \geq 0$, in the sense of quadratic forms, i.e. 
\begin{equation}\label{ineq1}
\int_\Omega |\nabla u|^2 dx \geq \int_{\Omega} W u^2 dx, \quad \forall u\in C_c^{\infty}(\Omega).
\end{equation} 	
\end{lem}    
This is a consequence of  Proposition \ref{identi} below. Indeed, by taking $\phi$ in Proposition \ref{identi} as in \eqref{id1} we get \eqref{ineq1} for test functions $u\in C_c^\infty(\Omega\setminus\{0\})$. The conclusion follows by showing the closure identity 
$$\overline{C_c^\infty(\Omega\setminus\{0\})}^{\|\cdot\|}=\overline{C_c^\infty(\Omega)}^{\|\cdot\|}, \quad \|u\|=\int_{\Omega}|\nabla u|^2 dx$$
(see e.g. \cite{CZ} for the details of this density argument).  
\begin{prop}\label{identi}
 Let $\phi$ be a positive function in $\Omega$ with $\phi \in C^2(\Omega\setminus\{0\})$.
	%
	Then it holds that
	\begin{equation}\label{geni}
	\int_{\Omega} \left( |\nabla u|^2 +\frac{\Delta \phi}{\phi}u^2 \right) dx =\int_\Omega \Big|\nabla
	u- \frac{\nabla \phi}{\phi} u\Big|^2 dx =\int_\Omega \phi^2|\nabla (u
	\phi^{-1})|^2 dx,
	\end{equation}
	for all $u \in C_c^\infty(\Omega\setminus \{0\})$.
\end{prop}

%
%

\begin{proof}[Proof of Proposition \ref{identi}]
	For a given  $u\in C_{c}^{\infty}(\Omega\setminus\{0\})$ we introduce the transformation  $u=\phi v$.
	Then we get
	$$|\nabla u|^2 = |\phi|^2 v^2 +\phi^2 |\nabla v|^2 +2\nabla \phi \cdot \nabla v \phi v. $$
	Applying  integration by parts we successively obtain
	\begin{align}\label{pema2}
	\int_\Omega |\nabla u|^2dx &=\int_\Omega |\nabla \phi|^2 v^2 dx + \int_\Omega \phi^2 |\nabla
	v|^2 dx +\frac{1}{2}\int_\Omega \nabla (\phi^2)\cdot  \nabla (v^2) dx\nonumber\\
	&=\int_\Omega |\nabla \phi|^2 v^2 dx + \int_\Omega \phi^2 |\nabla v|^2 dx-\frac{1}{2}
	\int_\Omega v^2 \mathrm{div}(2\phi \nabla \phi ) dx\nonumber\\
	&= \int_\Omega |\nabla \phi|^2 v^2 dx+  \int_\Omega \phi^2 |\nabla \phi|^2 dx -\frac{1}{2}
	\int_\Omega v^2 (2\phi \Delta \phi +2|\nabla \phi|^2)dx\nonumber\\
	& =\int_\Omega |\nabla v|^2 \phi^2 dx - \int_\Omega \phi \Delta \phi v^2 dx.
	\end{align}
	Therefore we finally have
	$$  \int_\Omega |\nabla u|^2 dx= \int_\Omega |\nabla v|^2 \phi^2 dx-\int_\Omega \frac{\Delta \phi}{\phi} u^2 dx.$$
	In conclusion we get  \eqref{geni}.
\end{proof}
Now we are in condition to "guess" the function $\phi(x)=|x|^{-(d-2)/2}$ in \eqref{extremal functions} which helped to prove the Hardy inequality \eqref{classical_Hardy_inequality}. We apply Lemma \ref{l1} for $W(x):=\mu/|x|^2$ and $\phi(x):= |x|^{\alpha}$ (it is quite natural to play with a radial function because the singularity has a radial structure) where $\alpha$ and $\mu$ will be precise later.   Computing, we have 
$$(-\Delta -W)\phi(x) = (-\alpha(\alpha+d-2)-\mu)|x|^{\alpha-2}, \quad \forall x\neq 0.$$
So, in view of Lemma \ref{l1}, if $-\alpha(\alpha+d-2)-\mu\geq 0$ we have
\begin{equation}\label{guess}
\int_\Omega |\nabla u|^2 dx\geq \mu \int_\Omega \frac{u^2}{|x|^2} dx. 
\end{equation}
With this argument, the biggest possible $\lambda$ in \eqref{guess} is 
$$\mu^\star=\max_{\alpha\in \mathbb{R}}\{-\alpha (\alpha +d-2)\}=\frac{(d-2)^2}{4},$$
which is obtained for $\alpha=-(d-2)/2$. Therefore we obtain the optimal pair $$(W(x), \phi(x))=\left(\frac{(d-2)^2}{4|x|^2}, |x|^{-(d-2)/2}\right),$$ 
and this  gives us another proof for the optimal Hardy inequality with interior singularity.

It turns out that when the singularity $x=0$ is located on the boundary $\partial \Omega$ of $\Omega$ then the best constant $\mu^\star(\Omega)$ improves with respect to the case of an interior singularity and its value depends on the global geometry of $\Omega$. 
 \begin{thm}[Boundary singularity]\label{th2} Let $d\geq 2$ and assume that $0\in \partial\Omega$. 
 \begin{enumerate}
 	\item\label{i1} 	Then 
	\begin{equation}\label{classical_Hardy_inequality2}
	\int_\Omega |\nabla u|^2 dx \geq \mu^\star(\Omega) \int_\Omega \frac{u^2}{|x|^2} dx, \quad \forall u\in C_0^\infty(\Omega),
	\end{equation}
	where
	\begin{equation}\label{global Hardy}
	\frac{(d-2)^2}{4}<\mu^\star(\Omega)\leq \frac{d^2}{4}.
	\end{equation}
		\item\label{i2} 	If $\Omega$ is a convex domain (or more general, contained in a half-space\footnote{see next page for the definition of the half-space $\mathbb{R}_+^d$ with respect to the last component $x_d$. The result remains valid for a half-space defined with respect to any of the components.}) then 
		$$\mu^\star(\Omega)=\frac{d^2}{4},$$ 
		which  is not attained in the energy space $\mathcal{D}^{1, 2}(\Omega)$.   
		\item\label{i3} For any domain $\Omega$ there exists $r>0$ depending on $\Omega$ (small enough) such that
		$$\mu^\star(\Omega \cap B_r(0))=\frac{d^2}{4}.$$
		\item\label{i4} 
		 If $\Omega$ is bounded there exists a constant $c\in \mathbb{R}$ such that  
	\begin{align}\label{oeq44}
	\quad H:=-\Delta-\frac{d^2}{4|x|^2}\geq c .
	\end{align}
	\item\label{i5}  If $\Omega$ is bounded and included in a half-space then \eqref{oeq44} holds for some $c>0$.\\
	\item\label{i6}  There exist domains  $\Omega$ for which $\mu^\star(\Omega)<d^2/4$.\\
	\item\label{i7}  If $\Omega$ is bounded and $\mu^\star(\Omega)< d^2/4$ then $\mu^\star(\Omega)$  is attained in $H_0^1(\Omega)$.\\
		\end{enumerate} 
\end{thm}

Theorem \ref{th2} reflects the fact that in the case of a boundary singularity  the best Hardy constant depends on the entire shape of the domain. 

Theorem \ref{th2} was completed in a series of works in the last 2 decades to whom we will refer in the following and sketch some ideas of the original proofs or alternative ones. 

\noindent {\it Short discussion on the results of Theorem \ref{th2}}.  First, it was proved in \cite{FTT} that $\mu^\star(\mathbb{R}_{+}^d)=d^2/4$ where $\mathbb{R}_{+}^d$ is the half-space in $\mathbb{R}^d$ defined by $\mathbb{R}_{+}^d:=\{x=(x_1, \ldots, x_d)\in \mathbb{R}^d \ |\ x_d>0 \}$. To our knowledge this  was subsequently extended to  domains included in a half-space (roughly speaking, this is item \eqref{i2}) independently in  the  works \cite{C_CRAS, C_arxiv} and  \cite{FM, Fall1}, respectively. The quoted authors also improved this by  proving Hardy-Poincar\'{e} type inequalities with positive reminder terms in bounded domains included in a half-space, which particularly imply item \eqref{i5}.   Probably the most surprising result of the theorem is item \eqref{i3} which was proved in \cite{Fall1}. It tells us that, locally near the singularity the best constant does not depend on the geometry of the domain.     As a consequence of this local result one can easily show the upper bound in \eqref{global Hardy} and item \eqref{i4} by  a localization argument using the partition of unity technique. The latter item says that the
constant $d^2/4$ is optimal, up to lower order terms in $L^2(\Omega)$-norm.  The idea of item \eqref{i6} is based on  approximations with conical domains near the singularity and it was proved in different presentations both in \cite{C_arxiv} and \cite{FM} with the help of the characterization of the Hardy constant in conical domains (see, e.g. \cite{PT}). The proof of  item \eqref{i7} was inspired from a proof originally from \cite{BM}. This and the non-attainability of the Hardy constant in the whole space imply the strict inequality in \eqref{classical_Hardy_inequality}. 

In the following we will give some details on item \eqref{i2} based on the pioneering result $\mu^\star(\mathbb{R}_+^d)=d^2/4$ proved in \cite{FTT} and the nice and unexpected result in item \eqref{i3} due to \cite{Fall1}. To simplify the presentation, our proofs below may be different in some aspects from the original proofs.

\begin{proof}[Sketch of proofs of items \eqref{i2} and \eqref{i3} in Theorem \ref{th2}] 
Assume first that $\mu^\star(\mathbb{R}_+^d)=d^2/4$. By homogeneity,  it is easy to see that the Hardy constant is invariant under dilatations, i.e. $\mu^\star(\Omega)=\mu^\star(\lambda\Omega)$, with $\lambda>0$, and invariant under rotations $T$ centered at $x=0$, i.e. $\mu^\star(\Omega)=\mu^\star(T(\Omega))$. Moreover, it is straightforward that $\mu^\star$ is nonincreasing with respect to set inclusion, i.e. for any $\Omega_1\subset\Omega_2$ we have $\mu^\star(\Omega_1)\geq \mu^\star(\Omega_2)$.  These facts ensure that $\mu(\mathbb{R}_+^d)=\mu^\star(B)$ for any ball $B$ containing the origin $x=0$ on its boundary. Then, just by comparison arguments one has that $\mu^\star(\Omega)=d^2/4$ for a domain $\Omega$ included in a half-space.

Now, let us prove the first result $\mu^\star(\mathbb{R}_+^d)=d^2/4$. 

For that we can apply Lemma \ref{l1}. Consider $W(x)=\mu/|x|^2$ and $\phi(x)=x_d|x|^{\alpha}$. Then, by direct computations we get 
$$\left(-\Delta -\frac{\mu}{|x|^2}\right)\phi(x)= (-\alpha(\alpha+d)-\mu)|x|^{\alpha-2}.$$
If we choose $\mu:=\max_{\alpha\in \mathbb{R}}\{-\alpha(\alpha+d)\}=d^2/4$ obtained by $\alpha=-d/2$, we deduce the admissible pair $(W(x), \phi(x))=(d^2/(4|x|^2), x_d |x|^{-d/2})$ which implies 
$$	\int_\Omega |\nabla u|^2 dx \geq \frac{d^2}{4} \int_\Omega \frac{u^2}{|x|^2} dx, \quad \forall u\in C_0^\infty(\Omega).$$
In consequence, $\mu^\star(\mathbb{R}_+^d)\geq d^2/4$. To prove the optimality we  just have to regularize and localize the solution $\phi(x)=x_d|x|^{-d/2}$ near the origin. For instance, we can choose as a minimizing sequence in the energy space $\mathcal{D}^{1, 2}(\mathbb{R}_+^d)$ (not necessary in $C_c^\infty(\mathbb{R}_+^d)$) the sequence 
\begin{equation}
u_\epsilon (x)= \left\{\begin{array}{ll}
x_d, & \textrm{ if } |x|\leq 1\\
x_d|x|^{-d/2+\epsilon}, & \textrm{ if } |x|\geq 1.\\
\end{array}\right.
\end{equation}
For the proof of item \eqref{i3} it suffices to consider $\Omega$ a domain for which the points on $\partial \Omega$ satisfy $x_d\leq 0$ in the neighborhood of the origin (e.g. $\Omega=\mathbb{R}^d\setminus B_1(-e_d)$, where $e_d=(0, \ldots ,0,1)$ is the $d$-th canonical vector of $\mathbb{R}^d$) and find a positive super-solution $\phi$ i.e. 
$$\left(-\Delta-\frac{d^2}{4|x|^2}\right)\phi(x)\geq 0,$$
in a small neighborhood $\Omega \cap B_{r}(0)$ of the origin $x=0$.    
		Such a super-solution was first built in \cite{Fall1} in terms of a local parametrization of the boundary $\partial \Omega$ near the origin  (it requires  tools like exponential
		maps and Fermi coordinates). In \cite[Appendix A]{C_JFA} we proposed a simplified expresion of the supersolution, i.e. 
		$$\phi(x)=\rho(x)|x|^{-d/2} e^{(1-d)\rho(x)}\left(\log \frac{1}{|x|}\right)^{1/2},$$
	where $\rho(x)=d(x, \partial \Omega):=\inf_{y\in \partial \Omega}|x-y|$ denotes the distance function from a point $x\in \overline{\Omega}$ to the boundary $\partial \Omega$.	Detailed computations can be found in \cite[Theorem 2.3.3]{PhDCristi}.     
\end{proof}

\subsection{The case $V(x)= \sum_{1\leq i< j\leq n} \frac{|a_i-a_j|^2}{|x-a_i|^2|x-a_j|^2}$}

Throughout this section we discuss the qualitative properties of
Schr\"{o}dinger (Hamiltonian) operators $-\Delta - \mu V(x)$, with inverse square potentials of the form
$$V(x):= \sum_{1\leq i< j\leq n} \frac{|a_i-a_j|^2}{|x-a_i|^2|x-a_j|^2},$$
for fixed configurations of singular poles $a_i\in \mathbb{R}^d$, $i=1, n$, with $n\geq 2$. 
We have been motivated by the previous works on multipolar potentials of type $\tilde V(x):=\sum_{i=1}^{n}
\alpha_i/|x-a_i|^2$ (with $\alpha_i\in\mathbb{R}$ and the singular poles $a_i\in \mathbb{R}^d$ being fixed)  which are associated with the interaction
of a finite number of electric dipoles where the interaction
among the poles depends on their relative partitions and the intensity of the singularity in
each of them (see, e.g. \cite{Levy}). They describe molecular systems such as the Hartree-Fock model (cf. \cite{Ca}) consisting of $n$ nuclei
of unit charge located at a finite number of points $a_1,\ldots, a_n$ and of $n$ electrons,  where Coulomb multi-singular potentials
arise in correspondence with the interactions between the electrons and the fixed nuclei.

In the case of the multi-singular potential $\tilde V(x)=\sum_{i=1}^{n}
\alpha_i/|x-a_i|^2$ , the study of positivity of
the quadratic functional
\begin{equation}\label{eq1}
\mathcal{T}
[u]:=\int_\Omega |\nabla u|^2 dx - \mu \int_\Omega
\tilde V(x)u^2 dx
\end{equation}
is much more intricate. To our knowledge, it is a challenging open problem even in the whole space $\mathbb{R}^d$ to determine the best constant $\mu^\star(\Omega, \tilde V)$  which makes $\mathcal{T}$ positive. 

Despite of that, some partial results are known. Particularly, in \cite{FT} it was proved  that when  
$\Omega=\mathbb{R}^d$ and $\mu=1$, $\mathcal{T}$ is positive if and only if
$\sum_{i=1}^{n}\alpha_i^{+}\leq (d-2)^2/4$ for any configuration of
the poles $a_1, \ldots, a_n$, where $\alpha^{+}=\max\{\alpha, 0\}$.
Conversely, if $\sum_{i=1}^{n}\alpha_i^{+}> (d-2)^2/4$, there exist
configurations
$a_1, \ldots, a_n$ for which $\mathcal{T}$ is negative (see also \cite{FTT} for complementary results). 
These results were improved later in \cite{BDE} where the authors showed that for any $\mu \in (0, (d-2)^2/4]$  and any
configuration $a_1, a_2, \ldots, a_n\in \mathbb{R}^d$, $n\geq 2$,
there is a nonnegative constant $K_n< \pi^2$ such that for all $u\in C_{c}^{\infty}(\mathbb{R}^d)$
\begin{equation}\label{este}
  \frac{K_n+(n+1) \mu}{M^2}\int_{\mathbb{R}^d}
u^2 dx+ \int_{\mathbb{R}^d} |\nabla u|^2 dx -\mu \sum_{i=1}^{n} \int_{\mathbb{R}^d}
\frac{u^2}{|x-a_i|^2}dx \geq 0,
\end{equation}
where $M$ denotes $M:=\min_{i\neq j}|a_i-a_j|/2$. The original proofs
of the above results require a partition of unity
technique, localizing the singularities and the classical Hardy inequality \eqref{classical_Hardy_inequality} in which the singularity $x=0$ is replaced with the singular poles $a_i$.  The "weak" inequality \eqref{este} emphasizes  that
we can reach the critical singular mass $(d-2)^2/(4|x-a_i|^2)$ at
any singular pole $a_i$ paying the prize of adding a lower order term in
$L^2$-norm with positive sign on the left hand side ("bad sign"). 

Besides, using the so-called "expansion of the square" method, the
authors in \cite{BDE} proved the multipolar inequality without lower
order terms with "bad sign":
\begin{align}\label{pushu}
\int_{\mathbb{R}^d} |\nabla u|^2 dx \geq \frac{(d-2)^2}{4 n^2}&  \int_{\mathbb{R}^d} V u^2 dx\nonumber\\
&+\frac{(d-2)^2}{4n} \sum_{i=1}^{n}\int_{\mathbb{R}^d} \frac{u^2}{|x-a_i|^2} dx, \quad \forall u\in C_c^\infty(\mathbb{R}^d),
\end{align}
for any fixed configuration $a_1, \ldots, a_n\in \mathbb{R}^d$, with $a_i \neq a_j$ for $i\neq j$.
Particularly, 
\begin{align}\label{pushu2}
\int_{\mathbb{R}^d} |\nabla u|^2 dx \geq \frac{(d-2)^2}{4 n^2}&  \int_{\mathbb{R}^d} V u^2 dx,  \quad \forall u\in C_c^\infty(\mathbb{R}^d).
\end{align}
Unfortunately, the constant $\mu^\star(\mathbb{R}^d, \tilde V)$ is not known  neither for $\tilde V$ which corresponds to $\alpha_i=1$ for all $i=1,n$.  We only know that $\mu^\star(\mathbb{R}^d, \tilde V)\geq (d-2)^2/(4n)$ which is a consequence of the Hardy inequality \eqref{classical_Hardy_inequality} applied for any singularity  $a_i$ but it is also visible from \eqref{pushu}. Nevertheless, motivated also by \eqref{pushu} we can make a compromise and analyze $\mu^\star(\Omega, \tilde V)$.    It occurs that the constant $(d-2)^2/(4n^2)$ in \eqref{pushu2} is not optimal.  

To answer to the optimality issue for $V$, we distinguish ans analyze two main cases: i) all the singularities of $V$ are in the interior of $\Omega$; ii) all the singularities of $V$ are located on the boundary $\partial \Omega$ of $\Omega$. We  first have
\begin{thm}[interior singularities]\label{1absprop1}
	Assume that $\Omega\subset \mathbb{R}^d$, $d\geq 3$, is a bounded domain such that  $a_1, \ldots, a_n\in \Omega$, $n\geq 2$.
\begin{enumerate}
	\item \label{j1}	It holds that
	\begin{equation}\label{better}
	\int_{\mathbb{R}^d} |\nabla u|^2 dx \geq \frac{(d-2)^2}{n^2} \int_{\mathbb{R}^d}
	V u^2 dx, \quad \forall u\in
	C_c^{\infty}(\mathbb{R}^d).
	\end{equation}
and the constant is optimal, i.e.  $\mu^\star(\mathbb{R}^d)=\frac{(d-2)^2}{n^2}.$
\item\label{j2} Besides,
	\begin{equation}\label{equality.constant}
	\left\{ \begin{array}{ll}
		\mu^\star(\Omega) =  \frac{(d-2)^2}{n^2}, & if\  n=2, \\ [10pt]
		\frac{(d-2)^2}{n^2} < \mu^\star(\Omega)\leq \frac{(d-2)^2}{4n-4}, &  if\  n\geq 3.\\
	\end{array}\right.
	\end{equation}
	Moreover, if $n=2$ then \eqref{equality.constant} is verified in any open domain $\Omega$ (not necessary bounded).
	\item \label{j3}  For any constant $\mu< (d-2)^2/(4n-4)$,
	there exists a finite constant  $c_\mu\in \mathbb{R}$ such that 
		$$-\Delta-\mu V \geq c_\mu.$$  
	\item\label{j4}  If $\mu^\star(\Omega)< (d-2)^2/(4n-4)$  then $\mu^\star(\Omega)$ is attained in $H_0^1(\Omega)$.
	\item\label{j5} If $n=2$  then $\mu^\star(\Omega)$  is never attained in $\mathcal{D}^{1, 2}(\Omega)$.
	\end{enumerate}
\end{thm}
Item \eqref{j1}  of Theorem \ref{1absprop1} was proved in \cite{CZ}. Later we extended this result and proved items \eqref{j2}-\eqref{j5} in \cite{C_CCM}. The second statement of item \eqref{j2} is the most surprising result of this theorem showing that in the case of interior singularities there is a gap between $\mu^\star(\Omega)$ and $\mu^\star(\mathbb{R}^d)$ when  $\Omega\subset \mathbb{R}^d$ is bounded and  $n\geq 3$. 

\begin{proof}[Sketch of proof of Theorem \ref{1absprop1}] Item \eqref{j1} is a consequence of the method of super-solutions, i.e. Lemma \ref{l1}. Indeed, we can check that 
	$$\left(W(x):=\frac{(d-2)^2}{n^2}V,  \quad \phi(x):=\prod_{i=1}^{n} |x-a_i|^{-(d-2)/n}\right)$$
	is an admissible pair since 
	$$\left(-\Delta- \frac{(d-2)^2}{n^2}V\right)\phi(x)=0, \quad \forall x\neq a_i.$$
We have to point out that Lemma \ref{l1} and Proposition \ref{identi} must be slightly modified because we have to avoid a finite number of singular points not just one. The rest of items follow more or less similar ideas from Theorem \ref{th2}.   We leave the details to the reader. 	
	\end{proof}

It is a popular fact that when switching from interior to boundary
singularities, the Hardy constant increases. It is also the case here. 

\begin{thm}[boundary singularities, cf. \cite{C_CCM}]\label{te2}
	Assume $\Omega\subset \mathbb{R}^d$, $d\geq 2$, is a domain  such that $a_1, \ldots, a_n \in \Gamma$, $n\geq 2$. In addition, for items \eqref{k2}-\eqref{k4} below we assume $\Omega$ to be bounded. 
	We successively have
	\begin{enumerate}
\item\label{k1} If $\Omega$ is either a ball, the exterior of a ball, or a  half-space in $\mathbb{R}^d$, $d\geq 2$ then 
		\begin{equation}\label{cond3}
		\mu^\star(\Omega)=\frac{d^2}{n^2}.
		\end{equation}
		If $\Omega$ is a ball, the constant $\mu^\star(\Omega)$ in \eqref{cond3} is attained in $H_0^1(\Omega)$ if and only if $n\geq 3$.\\
		If $\Omega$ is the exterior of a ball then $\mu^\star(\Omega)$ is attained in $\mathcal{D}^{1, 2}(\Omega)$ when $d\geq 3$ and $n\geq 3$, whereas if $\Omega$ is a half-space in $\mathbb{R}^d$ then $\mu^\star(\Omega)$ is attained in $\mathcal{D}^{1, 2}(\Omega)$ when $n\geq 3$.
		\item\label{k2} It holds that 
		\begin{equation}\label{bounds.bond}
		\frac{(d-2)^2}{n^2} <   \mu^\star(\Omega)\leq  \frac{d^2}{4n-4}.
		\end{equation}
		\item\label{k3}  For any constant $\mu< d^2/(4n-4)$,
		there exists  $c_\mu\in \mathbb{R}$ such that 
		$$-\Delta-\mu V \geq c_\mu.$$  
		\item\label{k4}  In addition, if $\mu^\star(\Omega)< d^2/(4n-4)$  then $\mu^\star(\Omega)$ is attained in $H_0^1(\Omega)$.
	\end{enumerate}
\end{thm}

\begin{proof}[Sketch of proof of Theorem \ref{te2}] 
	The main novelty of this theorem with respect to classical Hardy-type inequalities is concerned with the attainability of $\mu^\star(\Omega)$ in some particular cases and more precisely in balls. 
	Again, item \eqref{j1} is a consequence of Lemma \ref{l1} and Proposition \ref{identi} adapted to a finite number of singularities. Indeed, if $\Omega=B_r(c)$ is a ball of radius $r$ centered at a point $c\in \mathbb{R}^d$ we can check that 
	$$\left(W(x):=\frac{d^2}{n^2}V,  \quad \phi(x):=(r^2-|x-c|^2)\prod_{i=1}^{n} |x-a_i|^{-d/n}\right)$$
	is an admissible pair since 
	$$\left(-\Delta- \frac{d^2}{n^2}V\right)\phi(x)=0, \quad \forall x\neq a_i.$$ As we mentioned before, the novelty here is that the function $\phi$ belongs to the energy space $H_0^1(\Omega)$ and it is a minimizer of the best constant, i.e.  
	$$\frac{\int_\Omega |\nabla  \phi|^2 dx}{\int_\Omega V u^2 dx}=\frac{d^2}{n^2}. $$ 
 The case of the exterior of a ball and the half-space are very likely similar.  The rest of items follow  same ideas from Theorem \ref{1absprop1}. As before, the details are a good exercise for the reader. 	
\end{proof}

\section{The Rellich and Hardy-Rellich inequality in $L^2$ setting}

\subsection{The Rellich inequality} Developments of the Rellich inequality have emerged a lot in the last decades, especially when we refer to optimal results for weighted $L^p$-versions. The literature on the topic is huge but a minimal bibliography could be found for instance in \cite{O, DH, CM, Mo, Me} and the references therein.  The aim of this section is  to present a simple proof of the Rellich inequality based on by now, the powerful method of super-solutions. 

\begin{thm}[Rellich inequality]\label{rellich} Let $d\geq 5$.
For any $u\in C_c^\infty(\mathbb{R}^d)$ it holds
\begin{equation}\label{ineqRellich}
\int_{\mathbb{R}^d} |\Delta u|^2 dx \geq \frac{d^2(d-4)^2}{16} \int_{\mathbb{R}^d} \frac{u^2}{|x|^4} dx.
\end{equation}
Moreover, the constant $\mu^\star:= \frac{d^2(d-4)^2}{16}$ is optimal.
\end{thm}
In order to prove Theorem \ref{rellich} we need an auxiliary result.

\begin{prop}\label{auxresult}
Let $d\geq 5$ and $W\in L_{loc}^1(\mathbb{R}^d)$ such that $W$ is a continuous function in $\mathbb{R}^d\setminus\{0\}$.  Asssume $\phi$ is a $C^4$ function in $\mathbb{R}^d\setminus\{0\}$ such that
\begin{equation}\label{conditions}
\left\{
\begin{array}{ll}
\left(\Delta^2 - W \right)\phi(x) \geq 0, & \forall x\in \mathbb{R}^d\setminus\{0\}  \\
-\Delta \phi(x) \geq 0, &   \forall x\in \mathbb{R}^d\setminus \{0\}\\
\phi(x) >0,  & \forall x\in \mathbb{R}^d. \\
\end{array}
\right.
\end{equation}
Then
\begin{equation}\label{generalinequality}
\int_{\mathbb{R}^d} |\Delta u|^2 dx \geq \int_{\mathbb{R}^d}  W u^2 dx, \quad \forall u \in C_c^\infty(\mathbb{R}^d).
\end{equation}
\end{prop}

\begin{proof}
We employ the change of variables $u=\phi v$ and we obtain
\begin{align}\label{eq1}
\int_{\mathbb{R}^d} |\Delta u|^2 dx &= \int_{\mathbb{R}^d} \left(|\Delta \phi|^2 v^2 + \phi^2 |\Delta v|^2 + 4 |\nabla \phi \cdot \nabla v|^2+2 \phi v \Delta \phi \Delta v\right) dx \nonumber\\
&+  \int_{\mathbb{R}^d} \left(4 v \Delta \phi \nabla \phi \cdot \nabla v+ 4 \phi \Delta v \nabla \phi \cdot \nabla v - \frac{\Delta^2 \phi}{\phi} u^2  \right) dx.
\end{align}
On the other hand, integration by parts successively lead to
\begin{align}\label{eq2}
\int_{\mathbb{R}^d} \frac{\Delta^2 \phi}{\phi} u^2 dx &=\int_{\mathbb{R}^d} \phi v^2 \Delta^2 \phi= \int_{\mathbb{R}^d} \left(|\Delta \phi|^2 v^2 + 2 \phi v \Delta \phi \Delta v \right) dx \nonumber\\
&+ \int_{\mathbb{R}^d} \left(2 \phi \Delta \phi |\nabla v|^2 + 4 v \Delta \phi \nabla \phi \cdot \nabla v\right) dx.
\end{align}
Combining \eqref{eq1}-\eqref{eq2} we obtain
\begin{equation}\label{eq3}
\int_{\mathbb{R}^d} |\Delta u|^2 dx - \int_{\mathbb{R}^d} \left(\frac{\Delta^2 \phi}{\phi} \right)u^2 dx= \int_{\mathbb{R}^d} \left(|\phi\Delta v + 2 \nabla \phi \cdot \nabla v|^2- 2 \phi \Delta \phi |\nabla v|^2\right) dx,
\end{equation}
which is nonnegative.  Due to hypotheses \eqref{conditions} satisfied by $\phi$ we finally obtain \eqref{generalinequality}.
\end{proof}

\begin{proof}[Proof of Theorem \ref{rellich}]

We apply Proposition \ref{auxresult} for the pairs
 $(W(x), \phi(x))=(\mu/|x|^4, |x|^\alpha))$ where $\mu>0$ and $\alpha\in \mathbb{R}$ will be precise later. For such $\phi=\phi_\alpha$ we get the computations
$$\Delta \phi_\alpha= \left(\alpha^2 +\alpha(d-2) \right)|x|^{\alpha-2}, \quad \Delta^2 \phi_\alpha =\alpha (\alpha -2) (d-2 +\alpha) (d-4 +\alpha)|x|^{\alpha-4}.$$
Therefore, $\phi_\alpha$ verifies  \eqref{conditions} if $(\alpha, \lambda)$ is an \textit{admissible} pair, i.e. it verifies
\begin{equation}\label{admissible}
\left\{\begin{array}{ll}
\alpha (\alpha -2) (d-2 +\alpha) (d-4 +\alpha)-\mu\geq 0 & \\
\alpha^2+(d-2)\alpha\leq 0 & \\
\end{array}\right.
\end{equation}
The optimal Rellich inequality is the biggest $\mu>0$ for which there exists an admissible pair $(\alpha, \lambda)$ for \eqref{admissible}.  In view of these we consider the function $f:\mathbb{R} \rightarrow \mathbb{R}$ given by
$$f(\alpha)=\alpha (\alpha -2) (d-2 +\alpha) (d-4 +\alpha),$$ with the aim of maximizing $f$. Next we obtain the following table of variations. 
\begin{figure}[h]\caption{The variation of $f$}\label{fig1}
	\begin{center}
		\begin{tabular}{|c|  c  c  c  c  c|}
			\hline 
			$\alpha$  & 
			 $-\infty$ \quad  & $\alpha_1=\frac{-(d-4)-\sqrt{d^2-4d+8}}{2}$ \quad & $\alpha_2=-\frac{d-4}{2}$ \quad & $\alpha_3=\frac{-(d-4)+\sqrt{d^2-4d+8}}{2}$ &   $\infty$ \\ [1pt]
			\hline 
			$f'(\alpha)$ &  -  &  - 0  + & + 0 -  &  - 0 + & + \\ [1pt]
			\hline 
			$f(\alpha)$ & $\infty$ & $\searrow  f(\alpha_1) \nearrow$ &  $ \nearrow f(\alpha_2) \searrow $    & $\searrow f(\alpha_3) \nearrow$  &   $\infty$ \\  [1pt]
			\hline 
		\end{tabular}
	\end{center}
\end{figure}


On the other hand,  the second condition in \eqref{conditions} holds for $\phi_\alpha$ if and only if $\alpha\in [-(d-2), 0]$, thus we have to restrict the function $f$ to this interval and to maximize it. Finally, we get
\begin{equation}
\max_{\alpha\in [ -(d-2), 0]}f(\alpha)=f\left(-\frac{(d-2)}{4}\right)=\frac{d^2(d-4)^2}{16}. 
\end{equation}
Therefore $$\left(W(x)=\frac{d^2(d-4)^2}{16|x|^4},\quad  \phi(x)=|x|^{-\frac{d-4}{2}}\right)$$ is the admisible pair which provides the best constant $\mu^\star$. We have
$$\Delta^2 \phi-W\phi=\Delta^2 \phi-\frac{\mu^\star}{|x|^4}\phi=0, \quad \mu^\star=\frac{d^2(d-4)^2}{16}.$$
The optimality of $\mu^\star$ can be proved by a classical approximation argument ($\epsilon$-regularization of the extremal function $\phi(x)=|x|^{-\frac{d-4}{2}}$). Then the proof is completed. 
\end{proof}

\subsection{The Hardy-Rellich inequality}

It occurs that the Hardy-Rellich inequality stated in \eqref{First HR inequality}, which is trivially deduced by applying the Hardy inequality in the whole space, is not optimal.   In fact, we have
\begin{thm}[Hardy-Rellich inequality]\label{HardyBilaplacian} Let $d\geq 3$. Then 
	\begin{equation}\label{optconst} 
	\int_{\mathbb{R}^d} |\Delta u|^2 dx \geq \mu^\star(d) \int_{\mathbb{R}^d} \frac{|\nabla  u|^2}{|x|^2} dx, \quad u\in C_c^\infty(\mathbb{R}^d),
	\end{equation}
	with the optimal constant
	\begin{equation}
	\mu^\star(d)=\left\{\begin{array}{ll}
	\frac{d^2}{4}, & d\geq 5 \\[5pt]
	3, & d=4 \\[5pt]
	\frac{25}{36}, & d=3. \\[5pt]
	\end{array}\right.
	\end{equation}
\end{thm}
As usual, by optimal constant we understand
\begin{equation}\label{optim}
\mu^\star(d)=\inf_{u\in C_{c}^{\infty}(\mathbb{R}^d)\setminus \{0\}}\frac{\int_{\mathbb{R}^d}|\Delta u|^2 dx}{\int_{\mathbb{R}^d} |\nabla u|^2 /|x|^2dx}, 
\end{equation}
 
 To the best of our knowledge inequality \eqref{optconst} was firstly analyzed in the case of radial functions in \cite{Adi} where the authors showed that the best constant is $\mu^\star_{radial}(d)=d^2/4$ for any $d\geq 4$ (they do not give an answer for $d=3$). Soon after that Theorem \ref{HardyBilaplacian} was proved  in \cite{TZ} in higher dimensions $d\geq 5$ (the radial restriction was removed from \cite{Adi}). 
 The method in \cite{TZ} applies spherical harmonics decomposition but the proof fails for lower dimensions $d\in \{3, 4\}$.   This result was subsequently completed in lower dimensions $d\in \{ 3, 4\}$ independently in \cite{B} and  \cite{GM2} applying quite different techniques: Fourier transform tools, respectively  a quite general theory based on so-called Bessel pairs which allowed to obtain the most classical functional inequalities and their improvements in the literature.   It is also worth mentioning the work in \cite{HT} which complements the above papers with Rellich-type inequalities for vector fields. In the recent paper \cite{C_PRSE} we also refined the method implemented in \cite{TZ}, based on spherical harmonics decomposition, to give an easy and compact proof of the optimal Hardy-Rellich inequality \eqref{optconst} in any dimension $d\geq 3$. In addition, we provided minimizing sequences which were not explicitly mentioned in the previous quoted papers in lower dimensions $d \in \{3, 4\}$, emphasizing their symmetry breaking (see Theorem \ref{minimization} below).  In order to state the following theorem we need to introduce some preliminary facts.  
 First let us consider the Hilbert space $\mathcal{D}^{2,2}(\mathbb{R}^d)$ to be the completion of $C_c^{\infty}(\mathbb{R}^d)$ in the norm
 $$\|u\|=\left(\int_{\mathbb{R}^d} |\Delta u|^2 dx\right)^{1/2}. $$
 Of course, $\|\cdot\|$ is a norm on $C_c^\infty(\mathbb{R}^d)$ due to the weak maximum principle for harmonic functions.  In addition, we consider a smooth cut-off function $g\in C_c^\infty(\mathbb{R})$  such
 $$g(r)=\left\{\begin{array}{ll}
 1, & \textrm{ if } |r|\leq 1 \\
 0, & \textrm{ if } |r|\geq 2.
 \end{array}\right. $$

 \begin{thm}[Minimizing sequences]\label{minimization}
 	Let $\epsilon>0$ and define the sequence \begin{equation}\label{minseq}
 	u_\epsilon(x)=\left\{
 	\begin{array}{cc}
 	|x|^{-\frac{d-4}{2}+\epsilon}g(|x|), & \textrm{ if } d\geq 5 \\
 	|x|^{-\frac{d-4}{2}+\epsilon}g(|x|)\phi_1\left(\frac{x}{|x|}
 	\right),  & \textrm{ if } d\in \{3, 4\}, \\
 	\end{array}
 	\right.
 	\end{equation}
 	where $\phi_1$ is a spherical harmonic function of degree 1 such that $\|\phi_1\|_{L^2(S^{d-1})}=1$ ($S^{d-1}$ denotes the unit $(d-1)$-dimensional sphere centered at the origin in $\mathbb{R}^d$).
 	Then $\{u_\epsilon\}_{\epsilon>0}\subset \mathcal{D}^{2,2}(\mathbb{R}^d)$ is a minimizing sequence for $\mu^\star(d)$, i.e.
 	\begin{equation}\label{limit-seq}
 	\frac{\int_{\mathbb{R}^d}|\Delta u_\epsilon|^2 dx}{\int_{\mathbb{R}^d} |\nabla u_\epsilon|^2/|x|^2 dx }\searrow \mu^\star(d), \textrm{ as } \epsilon \searrow 0.
 	\end{equation}
 	Moreover,  the constant $\mu^\star(d)$ is not attained  in $\mathcal{D}^{2,2}(\mathbb{R}^d)$ (there are no minimizers in $\mathcal{D}^{2,2}(\mathbb{R}^d)$). 
 \end{thm}

Both Theorems \ref{HardyBilaplacian} and \ref{minimization} require non trivial proofs and we suggest to follows the quoted papers for their details. However, next we provide an easy proof of inequality  \eqref{optconst} in dimensions $d\geq 8$.
\begin{proof}[Proof of \eqref{optconst} for $d\geq 8$]
By integration by parts we get (starting from the mixed term)
\begin{equation}\label{identIP2}
(d-4)\int_{\mathbb{R}^d} \frac{|\nabla u|^2}{|x|^2} dx =-4 \int_{\mathbb{R}^d} \frac{|x\cdot \nabla u|^2}{|x|^4} dx + 2 \int_{\mathbb{R}^d} \frac{x\cdot \nabla u}{|x|^2} \Delta u dx.
\end{equation}
As a consequence of \eqref{identIP2} and Cauchy-Schwartz inequality we obtain
\begin{equation}\label{weakerineq}
\int_{\mathbb{R}^d} |\Delta u|^2 dx \geq \frac{d^2}{4}\int_{\mathbb{R}^d} \frac{|x\cdot \nabla u|^2}{|x|^4} dx, \quad \forall u \in C_c^\infty(\mathbb{R}^d),
\end{equation}
which holds for any $d\geq 4$ (we leave the few details to the reader). 

Let $\epsilon\in (0, 1/4]$ and $d\geq 4$. Then from \eqref{identIP2},  and the arithmetic inequality we get
\begin{align}\label{form}
(d-4)\int_{\mathbb{R}^d} \frac{|\nabla u|^2}{|x|^2} dx &\leq -4 \int_{\mathbb{R}^d} \frac{|x\cdot \nabla u|^2}{|x|^4} dx +\frac{1}{\epsilon} \int_{\mathbb{R}^d}  \frac{|x\cdot \nabla u|^2}{|x|^4} dx +\epsilon \int_{\mathbb{R}^d} |\Delta u|^2 dx\nonumber\\
&= \left[\left(\frac{1}{\epsilon} -4\right)\frac{4}{d^2}+\epsilon \right] \int_{\mathbb{R}^d} |\Delta u|^2 dx.
\end{align}
Let us consider the function $f:[0, 1/4]\rightarrow \mathbb{R}$ given by
$$f(\epsilon):=\left(\frac{1}{\epsilon} -4\right)\frac{4}{d^2}+\epsilon.$$
it is obvious that $f $  has as a critical point $\epsilon=2/d\in (0, 1/4]$ if $d\geq 8$. Therefore
$$\min_{\epsilon\in (0, 1/4]}f(\epsilon)=f\left(\frac 2 d \right)=\frac{4(d-4)}{d^2}$$
and taking $\epsilon=2/d$ in \eqref{form} we conclude the proof.

\end{proof}

\medskip

\noindent{\bf Acknowledgement.} The author has been partially founded by a grant of Ministry of Research and Innovation, CNCS-UEFISCDI Romania, project number PN-III-P1-1.1-TE-2016-2233, within PNCDI III.

\end{document}